\pdfoutput=1
\documentclass[a4paper,reqno,oneside]{article}

\usepackage{amsmath,amssymb,latexsym,color,epsfig,a4,enumerate}
\usepackage{a4wide}
\usepackage{amsmath,amsfonts,amssymb,amsthm}
\usepackage[utf8]{inputenc}
\usepackage{authblk}
\usepackage{ifthen}
\usepackage[style=american]{csquotes}
\usepackage{color}
\usepackage{graphicx}
\usepackage{relsize}
\usepackage{enumitem}

\newtheorem{theorem}{Theorem}[section]

\newtheorem{claim}[theorem]{Claim}

\newtheorem{corollary}[theorem]{Corollary}
\newtheorem{conjecture}[theorem]{Conjecture}

\newenvironment{prf}{\noindent{\bf Proof.\,}}{\hfill$\qed$}

\newcommand{\cR}{{\mathcal R}}
\newcommand{\cF}{{\mathcal F}}
\newcommand{\eps}{\varepsilon}

\begin{document}
%\date{\today}
%\title[An improved bound on the sizes of matchings guaranteeing a rainbow matching]{An improved bound on the sizes of matchings guaranteeing a rainbow matching}
\title{An improved bound on the sizes of matchings guaranteeing a rainbow matching}
\author{Dennis Clemens}
\author{Julia Ehrenm\"uller}
\affil{\footnotesize\textit{Technische Universität Hamburg-Harburg, Institut f\"ur Mathematik, Am Schwarzenberg-Campus 3, 21073 Hamburg, Germany, \{dennis.clemens, julia.ehrenmueller\}@tuhh.de}}
\date{\vspace{-0.5cm}}
%\author{Dennis Clemens}
%\address{(DC) Technische Universität Hamburg-Harburg, Institut f\"ur Mathematik, Am Schwarzenberg-Campus 3, 21073 Hamburg, Germany }
%\email{dennis.clemens@tuhh.de}

%\author{Julia Ehrenm\"uller}
%\address{(JE) Technische Universität Hamburg-Harburg, Institut f\"ur Mathematik, Am Schwarzenberg-Campus 3, 21073 Hamburg, Germany }
%\email{julia.ehrenmueller@tuhh.de}

\maketitle
\begin{abstract}
A conjecture by Aharoni and Berger states that every family of $n$ matchings of size $n+1$ in a bipartite multigraph contains a rainbow matching of size $n$. In this paper we prove that matching sizes of $\left(\frac 3 2 + o(1)\right) n$ suffice to guarantee such a rainbow matching, which is asymptotically the same bound as the best known one in case we only aim to find a rainbow matching of size $n-1$. This improves previous results by Aharoni, Charbit and Howard, and Kotlar and Ziv. 
\end{abstract}

\section{Introduction}
In this paper we are concerned with the question which sizes of $n$ matchings in a bipartite multigraph suffice in order to guarantee a rainbow matching of size $n$.

One motivation for considering these kinds of problems is due to some well known conjectures on Latin squares. 
A \emph{Latin square} of order $n$ is an $n \times n$ matrix in which each symbol appears exactly once in every row and exactly once in every column. A \emph{partial transversal} in a Latin square is a set of entries with distinct symbols such that from each row and each column at most one entry is contained in this set. 
We call a partial transversal of size $n$ in a Latin square of order $n$ simply \emph{transversal}. 
A famous conjecture of Ryser~\cite{ryser1967} states that for every odd integer $n$ any Latin square of order $n$ contains a transversal. 
The conjecture is known to be true for $n \leq 9$. 
Omitting the restriction to odd numbers yields a false statement. % as many counterexamples show~\cite{?}. 
Brualdi~\cite{brualdi1991, denes1974} and Stein~\cite{stein1975} independently formulated the following conjecture for all orders $n$.
\begin{conjecture}
\label{conj:brualdi}
For every $n \geq 1$ any Latin square of order $n$ has a partial transversal of size $n-1$. 
\end{conjecture}
A natural way to transfer this problem to graphs is the following. Let $L=(\ell_{i,j})_{i,j\in[n]}$ be a Latin square of order $n$. We define $G_{L}:=(A\cup B, E)$ as the complete bipartite edge-coloured graph with partite sets $A=\{a_1,\ldots,a_n\}$ and $B=\{b_1,\ldots,b_n\}$, where $a_ib_j$ is coloured $\ell_{i,j}$. That is, $A$ and $B$ represent the columns and rows of $L$, respectively. Moreover, a transversal of $L$ corresponds to a perfect matching in $G_L$ that uses each edge colour exactly once, which we call a \emph{rainbow matching} of size $n$. Using this notion, Conjecture~\ref{conj:brualdi} is equivalent to the following: For every $n\geq 1$ any complete bipartite edge-coloured graph, the colour classes of which are perfect matchings, contains a rainbow matching of size $n-1$. One may wonder whether this might even be true in the more general setting of bipartite edge-coloured multigraphs. 

Following Aharoni, Charbit and Howard~\cite{aharoni2015}, we define $f(n)$ to be the smallest integer $m$ such that 
every bipartite edge-coloured multigraph with exactly $n$ colour classes, each being a matching of size at least $m$,
contains a rainbow matching of size $n$. Aharoni and Berger~\cite{aharoni2009} conjectured
the following generalization of Conjecture~\ref{conj:brualdi}.
\begin{conjecture}
\label{conj:aharoni}
For every $n\geq 1$ we have $f(n) = n+1$. 
\end{conjecture}
The first approaches towards this conjecture are given by the bounds $f(n)\leq \left\lfloor \frac 7 4 n\right\rfloor$ due to Aharoni, Charbit and Howard~\cite{aharoni2015} and $f(n) \leq \left\lfloor \frac 5 3 n\right\rfloor$ due to Kotlar and Ziv~\cite{kotlar2014}. Here, we give an improved bound, which is asymptotically the same as the best known bound on the sizes of the colour classes in case we aim to find a rainbow matchings of size $n-1$~\cite{kotlar2014}. In particular, we prove the following.

\begin{theorem}
\label{thm:main}
For every $\eps >0$ there exists  an integer $n_0\geq 1$ such that for every $n \geq n_0$ we have $f(n)\leq \left(\frac{3}{2} + \eps\right)n$.
\end{theorem}

Subsequently, we use the following notation. Let $G$ be a bipartite multigraph with partite sets $A$ and $B$ and let $R$ be a matching in $G$. For a set $X \subseteq A$ we denote by $N_G(X|R):= \{y \in B: \exists xy\in R \text{ with } x \in X\}$ the neighbourhood of $X$ with respect to $R$. For the sake of readability, we omit floor and ceiling signs and do not intend to optimize constants in the proofs.

\section{Proof of Theorem~\ref{thm:main}}

In this section we give a proof of Theorem~\ref{thm:main} the idea of which can be summarized as follows. We start with assuming for a contradiction that a maximum rainbow matching in the given graph $G=(A\cup B,E)$ is of size $n-1$. A rainbow matching of this size is known to exist~\cite{kotlar2014}. We fix such a matching $R$ and find two sequence $e_1,\ldots,e_k$ and $g_1,\ldots,g_k$ of edges, the first consisting of edges from $R$ and the second consisting of edges outside $R$. 
We then show that either 
we can switch between some of the edges from the edge sequences to produce a rainbow matching of size $n$
(see the proofs of the Claims \ref{claim1}, 
\ref{claim2} and \ref{claim3}), 
or the matchings represented by the edges $e_1,\ldots,e_k$ need to touch at least $n$ vertices in $B$ that are saturated by $R$, both leading to a contradiction. To make the second case more precise we additionally introduce in the proof certain sequences
$X_1,\ldots,X_k\subseteq A$ and $Y_1,\ldots,Y_k\subseteq B$. 

\begin{prf}
%Let $\varepsilon>0$ be given, and whenever necessary, let us assume that $n$ is large enough. 
Let $\eps >0$ be given and whenever necessary we may
assume that $n$ is large enough. Let $\cF=\{F_0,\ F_1,\ \ldots,\ F_{n-1}\}$ be a family of $n$ matchings of size at least $(3/2 +\varepsilon)n$ in a bipartite multigraph $G=(A\cup B,E)$ with partite sets $A$ and $B$. We aim to find a rainbow matching of size $n$. 

For a contradiction, let us assume that there is no such matching. As shown in~\cite{kotlar2014}, there must exist a rainbow matching $R$ of size  $n-1$. We may assume without loss of generality that none of the edges of $F_0$ appears in $R$. 
Let $t$ be the smallest positive integer with $1/(2t-1)\leq \varepsilon$. Moreover, let
$X\subseteq A$ and $Y\subseteq B$ be the sets of vertices that are saturated by $R$, i.e.~incident with some edge of $R$.

In the following we show that for every $k \in [t]$ 
we can construct sequences
\begin{enumerate}[label=(S\arabic*)]
\item $e_1,\ldots,e_k$ of $k$ distinct edges $e_i=x_iy_i$ in $R$ with $x_i\in X$ and $y_i\in Y$,
\item $g_1,\ldots,g_k$ of $k$ distinct edges $g_i=z_iy_i$ with $z_i\in A\setminus X$,
\item $ X_1, \ldots, X_k$ of subsets of $X$,
\item $ Y_1, \ldots, Y_k$ of subsets of $Y$,
\end{enumerate}
and an injective function $\pi: \{0,1,\ldots,k\} \rightarrow \{0,1,\ldots,n-1\}$ with $\pi(0):=0$
such that the following properties hold:
\begin{enumerate}[label=(P\arabic*)]
\item\label{color_e} for each $ i\in [k]$ we have $e_i\in F_{\pi(i)}$,
\item\label{color_g} for each $ i\in [k]$ we have $g_i\in \bigcup_{j=0}^{i-1} F_{\pi(j)}$,
\item\label{disjoint} $(e_1\cup\ldots\cup e_k) \cap (X_k\cup Y_k)=\varnothing$,
\item\label{sizes} $|X_k|=|Y_k|=s_k:=2k\varepsilon n + k(7-3k)/2$,
\item\label{many_colors} for each $ i\in [k]$ and each $j\in\{0,\ldots,n-1\}$ it holds that if
$R$ contains an edge of the matching $F_j$ between $X_i$ and $Y_i$,
then there is also an edge of $F_j$ between $x_i$ and $B\setminus Y$,
\item\label{good_edges} for each $i\in [k]$ and each $w\in Y_i\setminus Y_{i-1}$ there exists a vertex $v\in A\setminus (X\cup \{z_1,\ldots,z_{i-1}\})$ such that $vw\in F_{\pi(i-1)}$ (where $Y_0 := \varnothing$), and
\item\label{different_endpoints} for each $i\in [k]$ and each $j\in [i-1]$ it holds that if $g_i\in F_{\pi(j)}$, then $z_i\in A\setminus (X\cup \{z_1,\ldots,z_j\})$.
\end{enumerate}

Before we start with the construction, let us first observe that by Property~\ref{sizes} we have a set $Y_t\subseteq Y$ which satisfies $2t\varepsilon n + t(7-3t)/2= |Y_t| \leq |Y| <n$. However, for large enough $n$ and by the choice of $t$ we have that $2t\varepsilon n + t(7-3t)/2 > n$, a contradiction. 

In order to find the sequences described above, we proceed by induction on $k$.
For the base case, let us argue why we find edges $e_1$, $g_1$, sets $X_1$, $Y_1$, and an injective function $\pi$ with Properties~\ref{color_e}-\ref{different_endpoints}. 
First observe that $F_0$ does not have any edges between $A\setminus X$ and $B \setminus Y$, by assumption on $R$. 
As $|F_0| \geq(3/2+\eps)n$, there are at least $(1/2+\eps)n+1$ edges of $F_0$ between $A\setminus X$ and $Y$. 
Let $N_0 \subseteq Y$ denote a set of size $(1/2+\eps)n+1$ such that for every vertex $w \in N_0$ there exists a vertex $v\in A\setminus X$ such that $vw \in F_0$. 
Furthermore, let $X_1':= N_G(N_0|R)$ and let $\cR_1 := \{F_j \in \cF: F_j \cap R[N_0, X_1']\neq \varnothing\}$. 

Let $F$ be any matching in $\cR_1$, let $vw$ be the unique edge in $R[N_0, X_1'] \cap F$ and let $z \in A\setminus X$ be the unique vertex such that $zw \in F_0$. 
Notice that there cannot be any edge $g$ of $F$ between $A\setminus (X \cup \{z\})$ and $B\setminus Y$,
since otherwise $(R\setminus \{vw\})\cup \{zw,g\}$ would give a rainbow matching of size $n$, in contradiction 
with $R$ being a maximum rainbow matching. 
Therefore, there are at least $(1/2+\eps)n + 1$ edges of $F$ between $B \setminus Y$ and $X \cup \{z\}$. 
Since $|X_1'| = (1/2+\eps)n+1$, there are at least $2\eps n +2$ edges of $F$ between $B\setminus Y$ and $X_1'$. 
Since this is true for any $F \in \cR_1$, we know by the pigeonhole principle that there is a vertex $x_1 \in X_1'$ and a subset $X_1 \subseteq X_1'$ of size $2\eps n +2$ such that, for every $F_j\in \cF $, if $F_j \cap R[X_1, N_G(X_1|R)]\neq \varnothing$ then $F_j$ has an edge between $x_1$ and $B\setminus Y$. Note that $x_1\notin X_1$. Let $e_1= x_1y_1$ be the unique edge in $R$ incident with $x_1$ and let $g_1 = z_1y_1$ be the unique edge of $F_0$ incident with $y_1 \in N_0$. Set $\pi(1)$ to the unique index $j\in [k]$ such that $e_1\in F_j$. One can easily verify that $e_1 = x_1y_1$, $g_1 = z_1y_1$, $X_1$, $Y_1 := N_G(X_1|R)$, and $\pi$ satisfy Properties~\ref{color_e}-\ref{different_endpoints}. 

For the induction hypothesis let us assume that for some $k \in [t-1]$ the above sequences are given with Properties~\ref{color_e}-\ref{different_endpoints}. We now aim
to extend these by edges $e_{k+1}, g_{k+1}$, sets $X_{k+1}, Y_{k+1}$, and a value $\pi(k+1)$ while maintaining
Properties~\ref{color_e}-\ref{different_endpoints}. We start with some useful claims.

\begin{claim}\label{claim1}
$F_{\pi(k)}$ has no edge between $A\setminus (X\cup \{z_1,\ldots,z_k\})$ and $B\setminus Y$.
\end{claim}

\begin{proof}[Proof of Claim~\ref{claim1}]
%{\bf Proof}
Assume for a contradiction that there exists an edge $g\in F_{\pi(k)}$ between the sets $A\setminus (X\cup \{z_1,\ldots,z_k\})$ and $B\setminus Y$. (See Figure~\ref{fig:cl1} for an illustration.) By Property~\ref{color_g} we find a sequence $k>j_1>j_2>\ldots>j_s=0$ with $1\leq s\leq k$ such that
\begin{align*}
g_k  & \in  F_{\pi(j_1)}\ , \\
g_{j_i} & \in  
F_{\pi(j_{i+1})} \text{\ \ for }i< s.
\end{align*}
Moreover, according to Property~\ref{different_endpoints} we know that $z_k,z_{j_1},\ldots,z_{j_{s-1}}$ are distinct,
and thus, also using Property~\ref{color_e}, we conclude that
$$(R\setminus \{e_k,e_{j_1},\ldots,e_{j_{s-1}}\})\cup \{g_k,g_{j_1},\ldots,g_{j_{s-1}},g\}$$
forms a rainbow matching which is larger than $R$, a contradiction.
\begin{figure} [htbp]
\centering
\includegraphics[scale=0.55]{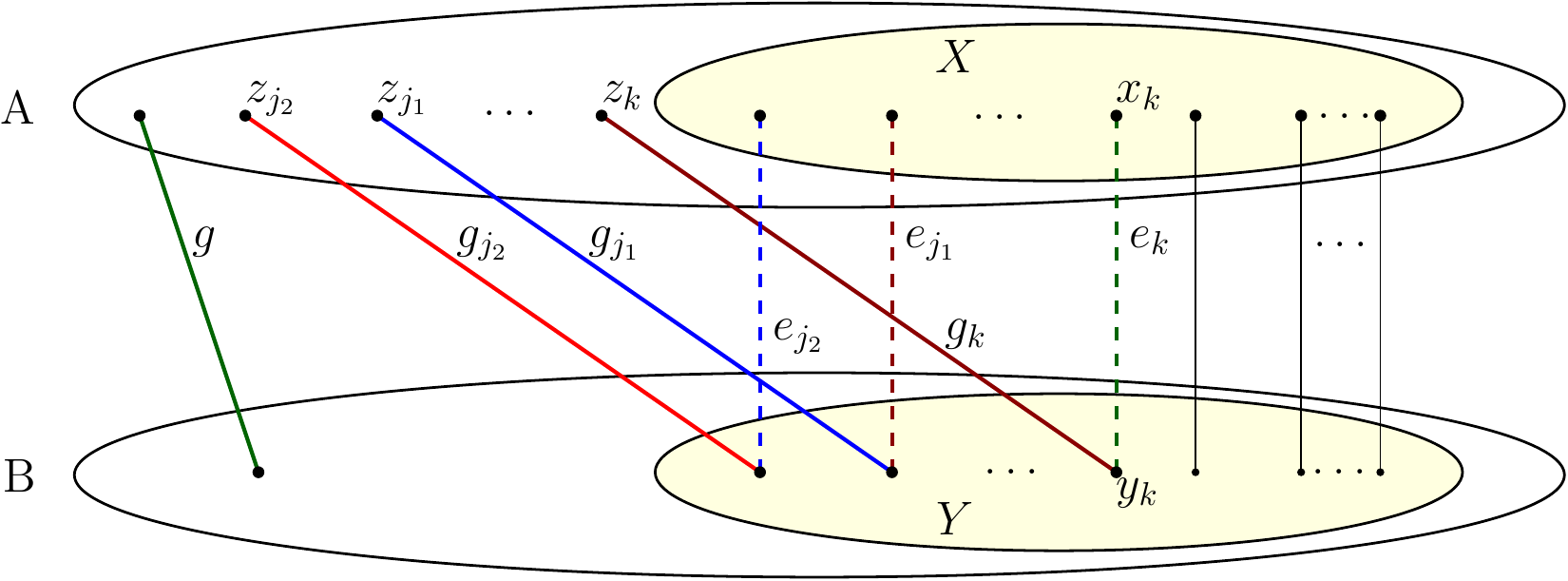}\label{fig:cl1}
 \caption{\em Example with $g_{j_2}\in F_{\pi(0)}$ ($s=3$). The dotted edges 
$\{e_k,e_{j_1},e_{j_2}\}$ are replaced by the edges $\{g_k,g_{j_1},g_{j_2},g\}$
to obtain a larger rainbow matching.}
\end{figure}
\end{proof}

\begin{claim}\label{claim2}
$F_{\pi(k)}$ has no edge between $A\setminus (X\cup \{z_1,\ldots,z_k\})$ and $Y_k$.
\end{claim}

\begin{proof}[Proof of Claim~\ref{claim2}]
Assume for a contradiction that there is an edge $g\in F_{\pi(k)}$ between the sets $A\setminus (X\cup \{z_1,\ldots,z_k\})$ and $Y_k$. 
(See Figure~\ref{fig:cl2} for an illustration.) 
Let $e$ be the unique edge in $R$ which is adjacent to $g$.
Observe that $e$ lies between $X_k$ and $Y_k$ by assumption. 
Let $j\in [n-1]$ be such that $e\in F_j$.
By Property~\ref{disjoint} we have $e\notin \{e_1,\ldots,e_k\}$. Thus, using Property~\ref{color_e} and the fact that $R$
is a rainbow matching, we can conclude that $j\notin \{\pi(i):1\leq i\leq k\}$. 
Now, by Property~\ref{many_colors} it holds that there is an edge $\overline{e}\in F_j$ between $x_k$ and $B\setminus Y$.
Moreover, by Properties~\ref{color_g} and~\ref{different_endpoints}, we
find a sequence $k>j_1>j_2>\ldots>j_s=0$ with $1\leq s\leq k$ such that
\begin{align*}
g_k  & \in  F_{\pi(j_1)}\ , \\
g_{j_i} & \in  
F_{\pi(j_{i+1})} \text{\ \ for }i< s
\end{align*}
and all vertices $z_k,z_{j_1},\ldots,z_{j_{s-1}}$ are distinct. Therefore, using Property~\ref{color_e}, we conclude that
$$(R\setminus \{e_k,e_{j_1},\ldots,e_{j_{s-1}},e\})\cup \{g_k,g_{j_1},\ldots,g_{j_{s-1}},\overline{e},g\}$$
forms a rainbow matching which is larger than $R$, a contradiction.
\end{proof}
\begin{figure} [htbp]
\centering
\includegraphics[scale=0.55]{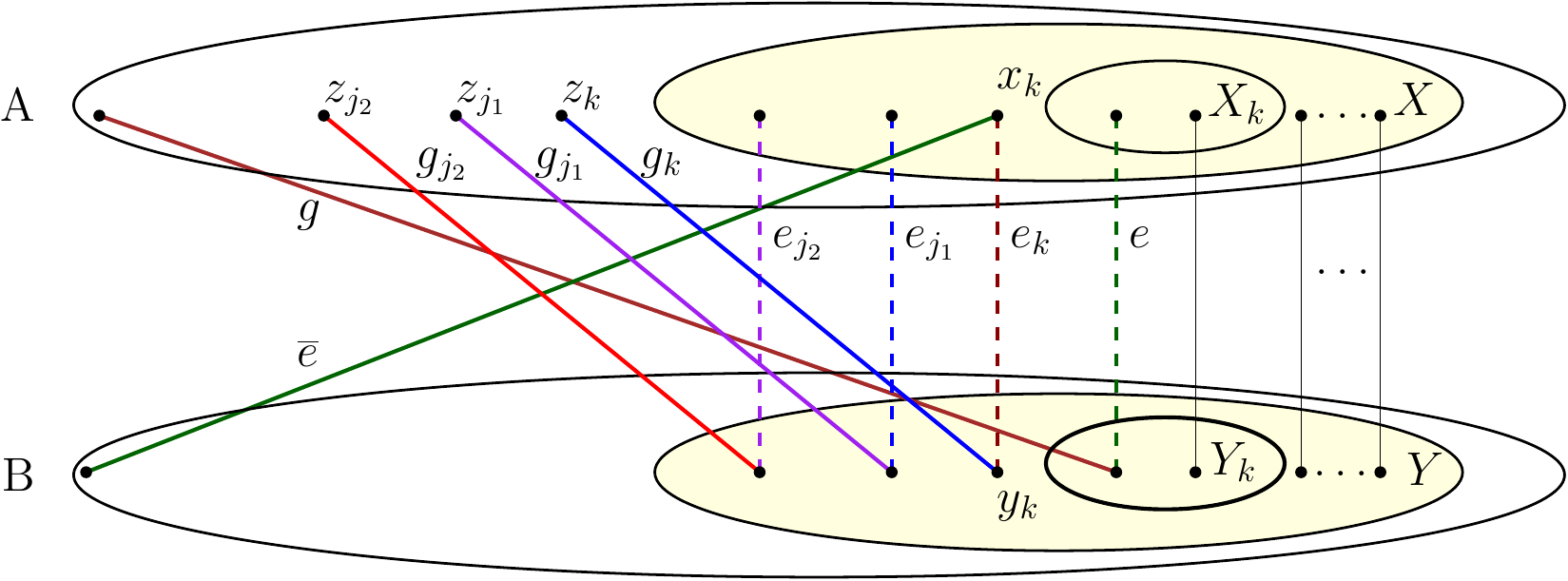}\label{fig:cl2}
 \caption{\em Example with $g_{j_2}\in F_{\pi(0)}$ ($s=3$). The dotted edges 
$\{e_k,e_{j_1},e_{j_2},e\}$ are replaced by the edges $\{g_k,g_{j_1},g_{j_2},\overline{e},g\}$
to obtain a larger rainbow matching.}
 \end{figure}

\begin{corollary}\label{existence_Nk}
The matching $F_{\pi(k)}$ has at least $\left(\frac{1}{2}+\varepsilon\right)n+1-2k$ edges between \linebreak
$A\setminus (X\cup \{z_1,\ldots,z_k\})$ and $Y\setminus (Y_k\cup \{y_1,\ldots,y_k\})$.
\end{corollary}

\begin{proof}
As $|F_{\pi(k)}|\geq (3/2+\varepsilon)n$ and $|X\cup \{z_1,\ldots,z_k\}|\leq n-1+k$,
we conclude that at least $(1/2+\varepsilon)n+1-k$ edges of $F_{\pi(k)}$
are incident with vertices in $A\setminus (X\cup \{z_1,\ldots,z_k\})$. Each of these edges
intersects $Y\setminus Y_k$ by the previous claims
and thus the statement follows.
\end{proof}

In the following, let $N_k\subseteq Y\setminus (Y_k\cup \{y_1,\ldots,y_k\})$ be a set of size
$1/2+\varepsilon)n+1-2k$ such that for each vertex $w\in N_k$
there is a vertex $v\in A\setminus (X\cup \{z_1,\ldots,z_k\})$ with $vw\in F_{\pi(k)}$.
Such a set exists by the previous corollary.
Moreover, let $$Y_{k+1}':=Y_k\cup N_k$$
and let $X_{k+1}':= N_G(Y_{k+1}'|R)$ be the neighbourhood of $Y_{k+1}'$ with respect to $R$.
By Property~\ref{sizes}, and as $N_k\cap Y_k=\varnothing$, we obtain
\begin{align}
|X_{k+1}'|=|Y_{k+1}'|&=2k\varepsilon n + \frac{k(7-3k)}{2} + \left(\frac{1}{2}+\varepsilon\right)n+1-2k \nonumber \\
& = \frac{1}{2}n + (2k+1)\varepsilon n + \frac{-3k^2+3k+2}{2}\ . \label{sizeXY}\tag{$\ast$}
\end{align}

We now look at all matchings that have an edge in $R$ between $X_{k+1}'$ and $Y_{k+1}'$.
Formally, we consider
$$\cR_{k+1}:=\big\{F_j\in \cF: F_j\cap R[X_{k+1}',Y_{k+1}']\neq \varnothing\big\}\ .$$

\begin{claim}\label{claim3}
Every $F_j\in \cR_{k+1}$ has at least $s_{k+1}$
edges between $X_{k+1}'$ and $B\setminus Y$.
\end{claim}

\begin{proof}
The main argument is similar to that of Claim \ref{claim1} - Corollary \ref{existence_Nk}.
For $F_j\in \cR_{k+1}$  let $f=vw$, with $v\in X_{k+1}',\ w\in Y_{k+1}'$, denote the unique edge
in  $F_j\cap R[X_{k+1}',Y_{k+1}']$.
Since $Y_{k+1}':=Y_k\cup N_k$, we either have $w\in Y_k$ or $w\in N_k$.
In particular, by Property~\ref{disjoint} from the hypothesis and by the definition of $N_k$,
we know that $w\notin \{y_1,\ldots,y_k\}$, and therefore $j\notin \{\pi(i): 0\leq i\leq k\}$.

If $w\in Y_k$, then we find an integer $j_1\in [k]$ such that $w\in Y_{j_1}\setminus Y_{j_1-1}$
since $Y_k=\bigcup_{i\in [k]} Y_i\setminus Y_{i-1}$,
and by Property~\ref{good_edges} there is a vertex $z\in A\setminus (X\cup \{z_1,\ldots,z_{j_1-1}\})$
such that $zw\in F_{\pi(j_1-1)}$.\\
If otherwise $w\in N_k$,
we find a vertex  $z\in A\setminus (X\cup \{z_1,\ldots,z_k\})$
such that $zw\in F_{\pi(k)}$, by construction of $N_k$. In either case, let us fix this particular vertex $z$.
We now prove the claim by showing first that (i) $F_j$ has no edge between $A\setminus (X\cup \{z_1,\ldots,z_k,z\})$ and $B\setminus Y$, and then we conclude that (ii) the statement holds for $F_j$.

We start with the discussion of~(i). So, assume that
$F_j$ has an edge $\overline{f}$ between $A\setminus (X\cup \{z_1,\ldots,z_k,z\})$ and $B\setminus Y$.

If $w\in Y_k$, then by the definition of $z$ we have $zw\in F_{\pi(j_1-1)}$, with $j_1$ being defined above.
We can assume that $j_1> 1$, as otherwise $zw\in F_0$ and thus 
$(R\setminus \{f\})\cup \{\overline{f},zw\}$ forms a full rainbow matching, in contradiction to our main assumption.
But then, using Property~\ref{color_g}, we find a sequence $j_1-1>j_2>\ldots >j_s=0$ with $2\leq s<k$ such that 
\begin{align*}
g_{j_1-1}  & \in  F_{\pi(j_2)}\ , \\
g_{j_i} & \in  F_{\pi(j_{i+1})} \text{\ \ for }2\leq i\leq s-1
\end{align*}
and, by Property~\ref{different_endpoints} and since $z\in A\setminus (X\cup \{z_1,\ldots,z_{j_1-1}\})$,
all the vertices $z,z_{j_1-1},z_{j_2},\ldots,z_{j_{s-1}}$ are distinct. We thus find the rainbow matching
$$(R\setminus \{e_{j_1-1},e_{j_2},\ldots,e_{j_{s-1}},f\})\cup \{g_{j_1-1},g_{j_2},\ldots,g_{j_{s-1}},\overline{f},zw\}$$
which is larger than $R$, a contradiction.

\begin{figure} [htbp]
\centering
\includegraphics[scale=0.55]{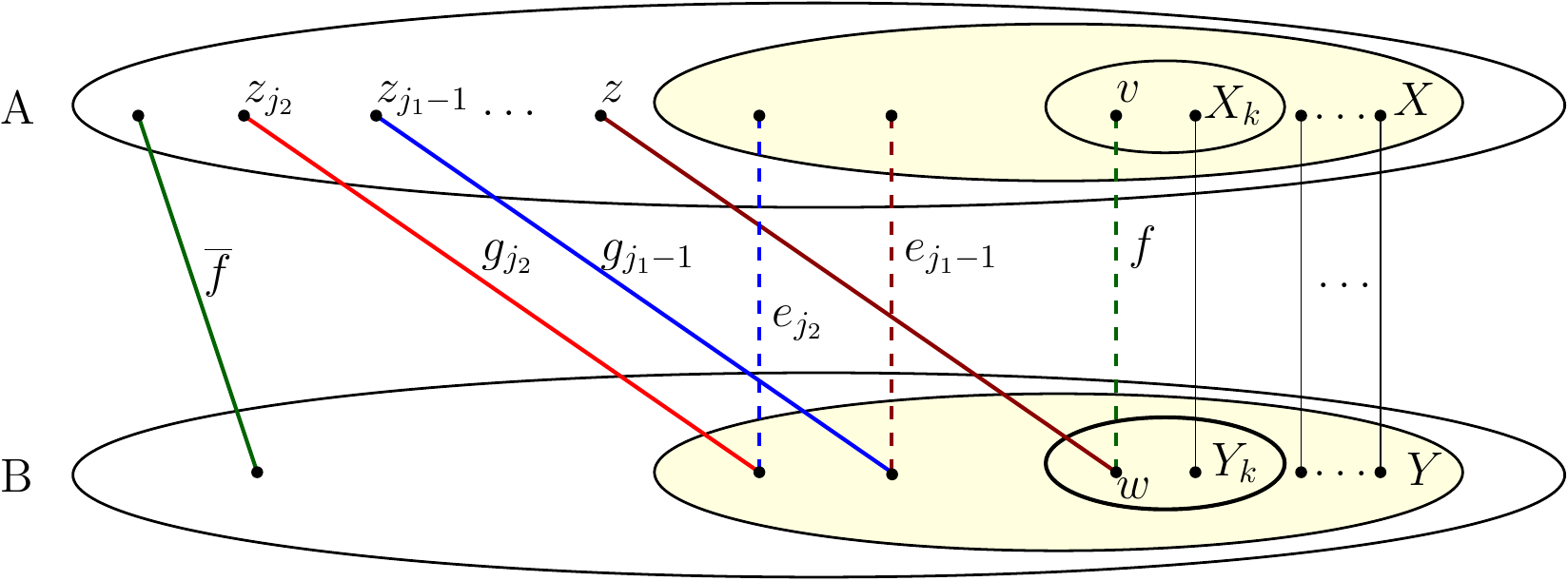}\label{fig:cl3a}
 \caption{\em Example with $g_{j_2}\in F_{\pi(0)}$, in case $w\in Y_k$. The dotted edges 
$\{e_{j_1-1},e_{j_2},f\}$ are replaced by the edges $\{g_{j_1-1},g_{j_2},\overline{f},zw\}$
to obtain a larger rainbow matching.}
 \end{figure}

If otherwise $w\in N_k$, then $zw\in F_{\pi(k)}$. Analogously we find
a sequence $k>j_1>j_2>\ldots>j_s=0$ with $1\leq s\leq k$ such that
$g_{k} \in  F_{\pi(j_1)}$ and $g_{j_i} \in F_{\pi(j_{i+1})}$ for $i<s$,
and we obtain a contradiction as
$$(R\setminus \{e_{k},e_{j_1},\ldots,e_{j_s},f\})\cup \{g_{k},g_{j_1},\ldots,g_{j_s},\overline{f},zw\}$$
forms a rainbow matching which is larger than $R$. Thus, we are done with part (i).

Let us proceed with (ii): $F_j$ needs
to saturate at least $(1/2 + \varepsilon )n+1$ vertices of $B\setminus Y$,
as $|F_j|\geq (3/2 + \varepsilon )n$ and $|Y|\leq n-1$.
Thus, by part (i),
we have at least $(1/2 + \varepsilon )n+1$ edges of $F_j$
between $X\cup \{z_1,\ldots,z_k,z\}$ and $B\setminus Y$. Using (\ref{sizeXY}), we further calculate
that
\begin{align*}
|X\cup \{z_1,\ldots,z_k,z\}|-|X_{k+1}'| & \leq (n+k) - \left(\frac{1}{2}n + (2k+1)\varepsilon n + \frac{-3k^2+3k+2}{2} \right)\\
& = \frac{1}{2}n - (2k+1)\varepsilon n + \frac{3k^2-k-2}{2} \ .
\end{align*}
Thus, the number of edges in $F_j$ between $X_{k+1}'$ and $B\setminus Y$ needs to be at least
\begin{align*}
\left( \frac{1}{2} + \varepsilon \right)n+1 - \left( \frac{1}{2}n - (2k+1)\varepsilon n + \frac{3k^2-k-2}{2} \right) = s_{k+1}\ ,
% = 2(k+1)\varepsilon n + \frac{(k+1)(7-3(k+1))}{2}
\end{align*}
as claimed.
\end{proof}

We now proceed with the construction of the edges  $e_{k+1}, g_{k+1}$ and the sets $X_{k+1}, Y_{k+1}$,
and afterwards we show that all required properties are maintained. The next corollary is by the pigeonhole
principle an immediate consequence of Claim~\ref{claim3}.

\begin{corollary}\label{sequence}
There exists a vertex $x_{k+1}\in X_{k+1}'$, a set $X_{k+1}\subseteq X_{k+1}'$ of size $s_{k+1}$
and its neighborhood $Y_{k+1}\subseteq Y_{k+1}'$ with respect to $R$ such that the following holds
for every $j\in[n-1]$: 
If $F_j\cap R[X_{k+1},Y_{k+1}]\neq \varnothing$, then $F_j$ has an edge
between $x_{k+1}$ and $B\setminus Y$. \hfill \qed
\end{corollary}

To extend the sequences, choose $X_{k+1}$ and $Y_{k+1}$ according to Corollary \ref{sequence},
and let $e_{k+1}=x_{k+1}y_{k+1}$ be the unique edge in $R$ that is incident with $x_{k+1}$.
Note that $x_{k+1}\notin X_{k+1}$, as otherwise $x_{k+1}$ would need to be incident to two edges
of the same matching $F_j$.

Observe that $y_{k+1}\notin \{y_1,\ldots,y_k\}$. Indeed, $y_{k+1}\in Y_{k+1}'=Y_k\cup N_k$,
and by construction we have $N_k\cap \{y_1,\ldots,y_k\}=\varnothing$,
while $Y_k\cap \{y_1,\ldots,y_k\}=\varnothing$ holds by Property~\ref{disjoint}.

Now, let $e_{k+1}\in F_j$. As $e_{k+1}\in R\setminus \{e_1,\ldots,e_k\}$, we have
$j\notin \{\pi(i):\ 0\leq i\leq k\}$. We extend the injective function $\pi$ with $\pi(k+1)=j$.

Finally, we choose $g_{k+1}$ as follows: 
If $y_{k+1}\in N_k$, then by construction of $N_k$ there is a vertex $z_{k+1}\in A\setminus (X\cup \{z_1,\ldots,z_k\})$
with $z_{k+1}y_{k+1}\in F_{\pi(k)}$. Otherwise, if $y_{k+1}\in Y_k$,
then there is an $i\in [k]$ with $y_{k+1}\in Y_i\setminus Y_{i-1}$, and by Property~\ref{good_edges}
there is a vertex $z_{k+1}\in A\setminus (X\cup \{z_1,\ldots, z_{i-1}\})$ such that $z_{k+1}y_{k+1}\in F_{\pi(i-1)}$.
In any case, we set $g_{k+1}:=z_{k+1}y_{k+1}$.

\begin{claim}
\label{claim:properties}
The extended sequences satisfy Properties~\ref{color_e}-\ref{different_endpoints}.
\end{claim}

\begin{proof}
 Properties~\ref{color_e} and~\ref{color_g} follow immediately from the induction hypothesis and from the definition of $\pi(k+1)$ and $g_{k+1}$. 
 By construction, we have $Y_{k+1}\subseteq Y_{k+1}' = Y_k \cup N_k$. 
 By Property~\ref{disjoint} of the induction hypothesis and by the definition of $N_k$, we have $\{y_1,\ldots,y_k\} \cap Y_{k+1} = \varnothing$. 
 It follows from the construction of $X_{k+1}$ (Corollary~\ref{sequence}) that $y_{k+1} \notin Y_{k+1}$. 
 By symmetry, we have $\{e_1, \ldots, e_{k+1}\}\cap(X_{k+1}\cup Y_{k+1}) = \varnothing$, which shows Property~\ref{disjoint}. 
 Properties~\ref{sizes} and~\ref{many_colors} hold by Corollary~\ref{sequence} and by Property~\ref{many_colors} of the induction hypothesis. 
 Recall that $Y_{k+1} \setminus Y_{k} \subseteq N_k$. 
 This means that for every $w\in Y_{k+1}\setminus Y_k$ there exists a vertex $v \in A\setminus (X\cup \{z_1,\ldots, z_{k}\})$ such that $vw \in F_{\pi(k)}$, proving Property~\ref{good_edges}. 
 Finally, Property~\ref{different_endpoints} holds by the induction hypothesis and since we chose $z_{k+1}$ from a set $A\setminus (X\cup \{z_1,\ldots, z_{i-1}\})$ such that $z_{k+1}y_{k+1}\in F_{\pi(i-1)}$ for the appropriate $i \in [k+1]$. 
 Consequently, all Properties~\ref{color_e}-\ref{different_endpoints} are fulfilled by the extended sequences. 
\end{proof}
Claim~\ref{claim:properties} concludes the induction and thus the proof of Theorem~\ref{thm:main}. 
\end{prf}

\section{Open problems and concluding remarks}

In this paper we proved that a collection of $n$ matchings of size $\left(3/2 + o(1)\right)n$ in a bipartite multigraph guarantees a rainbow matching of size $n$. 
One of the obstacles why our proof does not work for smaller values is that it is not clear what matching sizes are sufficient for guaranteeing a rainbow matching of size $n-1$. 
More generally, as suggested by Tibor Szabó (private communication), it would be interesting to determine upper bounds on the smallest integer $\mu(n,\ell)$ such that every family of $n$ matchings of size $\mu(n,\ell)$ in a bipartite multigraph guarantees a rainbow matching of size $n-\ell$. One can verify that $\mu(n,l) \leq \frac{l+2}{l+1} n$. Moreover, it holds that $\mu(n, \sqrt{n})\leq n$, which is a generalization (see e.g.~\cite{aharoni2013}) of a result proved in the context of Latin squares by Woolbright~\cite{woolbright}, and independently by Brouwer, de Vries and Wieringa~\cite{brouwer1978}. 
%It would be nice to improve that result and find a $k=o(\sqrt{n})$ such that $\mu(n,k)\leq n$. 

In order to approach Conjecture~\ref{conj:aharoni}, one can also increase the number of matchings and fix their sizes to be equal to $n$ instead of considering families of $n$ matchings of sizes greater than $n$. 
Drisko~\cite{drisko1998} proved that a collection of $2n-1$ matchings of size $n$ in a bipartite multigraph with partite sets of size $n$ guarantees a rainbow mathching of size $n$. He also showed that this result is sharp. This problem can be further investigated in the following two directions. Does the statement also hold if we omit the restriction on the sizes of the vertex classes? And how many matchings do we need to find a rainbow matching of size $n-\ell$ for every $\ell \geq 1$? 

Finally, in case Conjecture~\ref{conj:aharoni} turns out to be true, it is of interest to see how sharp it is.
As shown by Barat and Wanless~\cite{barat2014}, one can find constructions of $n$ matchings with $\left\lfloor \frac{n}{2}\right\rfloor -1$ matchings of size $n+1$ and the remaining ones being of size $n$ such that there is no rainbow matching of size $n$. We wonder whether the expression
$\left\lfloor \frac{n}{2}\right\rfloor -1$ above could also be replaced by $(1-o(1))n$.

\bibliographystyle{abbrv} 	
\bibliography{rainbowmatchings}	

\end{document}